\documentclass{article}
\usepackage{graphicx} % Required for inserting images

\usepackage[utf8]{inputenc}
\usepackage{amsmath,amsthm,verbatim,amssymb,amsfonts,amscd, graphicx}
\usepackage{graphics,caption}
\usepackage{hyperref}
\hypersetup{
    colorlinks=true,   	% false: boxed links; true: colored links
    linkcolor=red,      % color of internal links (change box color 
    citecolor = [rgb]{0 0.7 0},   	% color of links to bibliography
    filecolor=magenta, 	% color of file links
    urlcolor=blue
}
\usepackage{enumitem}
\usepackage{apptools}
\usepackage{titlesec}
\usepackage{comment}
\usepackage{color,bm}
\usepackage[baseline]{euflag}

\title{Finite de Finetti bounds in relative entropy}
\author{Lampros Gavalakis\thanks{Univ Gustave Eiffel, 
	Univ Paris Est Creteil, CNRS, 
	LAMA UMR8050 F-77447 Marne-la-Vall{\'e}e, France. 
	Email: {\tt lampros.gavalakis@univ-eiffel.fr}. 
	L.G. has received funding from the European Union's Horizon 2020 
	research and innovation program
 	under the Marie Sklodowska-Curie grant agreement 
	No 101034255 {\euflag} and by the B{\'e}zout Labex, funded by ANR, 
	reference ANR-10-LABX-58.}
 \and
 Oliver Johnson\thanks{School of Mathematics, 
	University of Bristol, Woodland Road, Bristol
	BS8 1UG, U.K.
	Email: {\tt O.Johnson@bristol.ac.uk}.
	} 
\and 
Ioannis Kontoyiannis%
\thanks{Statistical Laboratory, DPMMS,
	University of Cambridge,
	Centre for Mathematical Sciences,
        Wilberforce Road,
	Cambridge CB3 0WB, U.K.
        Email: {\tt yiannis@maths.cam.ac.uk}.
}
}
\date{\today}

\newcommand{\tv}[2]{
\|#1-#2\|_{\mathrm{TV}}
}

\newcommand{\MS} {
\mathcal{M}(S)
}

\newcommand{\spro}{\mbox{\boldmath $s$}}
\newcommand{\ellp}{\mbox{\boldmath $\ell$}}
\newcommand{\ellps}{\mbox{\scriptsize\boldmath $\ell$}}

\newcommand{\RL}{{\mathbb R}}

\newcommand{\BBP}{{\mathbb P}}

\def\ba{\begin{align}}
\def\ea{\end{align}}
\def\ban{\begin{align*}}
\def\ean{\end{align*}}

\def\be{\begin{eqnarray}}
\def\ee{\end{eqnarray}}
\def\ben{\begin{eqnarray*}}
\def\een{\end{eqnarray*}}

\def\bqq{\begin{equation}}
\def\eqq{\end{equation}}
\def\bqqn{\begin{equation*}}
\def\eqqn{\end{equation*}}

\def\sq{$\Box$}

\def\qed{\ifmmode\sq\else{\unskip\nobreak\hfil
\penalty50\hskip1em\null\nobreak\hfil\sq
\parfillskip=0pt\finalhyphendemerits=0\endgraf}\fi\par\medbreak}

\newsavebox{\junk}
\savebox{\junk}[1.6mm]{\hbox{$|\!|\!|$}}

\def\til={{\widetilde =}}

\def\clA{{\cal A}}
\def\clB{{\cal B}}

\def\clF{{\cal F}}

\def\clM{{\cal M}}

\def\clS{{\cal S}}

 \def\eq#1/{(\ref{#1})}

\newtheorem{theorem}{Theorem}[section]

\newtheorem{lemma}[theorem]{Lemma}

\def\eq#1/{(\ref{e:#1})}

\def\bdes{\begin{description}}
\def\edes{\end{description}}

\def\notes#1{}

\definecolor{mag}{rgb}{0.7,0,0.3}
\definecolor{dgreen}{rgb}{0.1,0.5,0.1}
\definecolor{dred}{rgb}{.8,0,0}
\definecolor{gray}{rgb}{.8,.8,.8}
\definecolor{brown}{rgb}{0.6451,0.3706,0.1745}

\begin{document}

\maketitle

\begin{abstract}
    	We review old and recent finite de Finetti theorems in total 
	variation distance and in relative entropy,
	and we highlight their connections with bounds 
	on the difference between sampling 
	with and without replacement. 
	We also establish two new finite de Finetti theorems
	for exchangeable random vectors 
	taking values in arbitrary spaces.
	These bounds are tight, and they are 
	independent of the size and the dimension 
	of the underlying space.
\end{abstract}

\vspace{0.5in}

\section{Information in probability}
%%%%%%%%%%%%%%%%%%%%%%%%%%%%%%%%%%%%%%%%%%%%%%%%%%%%%%%%%%%%%%%%%%%%%%%%%%

\noindent 
{\bf Early history. }
Classical information theory is built on probability theory. Not only 
is its core toolbox and vocabulary probabilistic but, since its early days,
sophisticated probabilistic techniques were effectively utilised
in many core information-theoretic problems.
For example, what is arguably the crowning achievement of information theory 
-- Shannon's channel coding theorem~\cite{shannon:48} -- is proved 
with one of the earliest applications of 
the ``probabilistic method''~\cite{alon-spencer:book}.
Since then, there has been a consistent influx of 
modern probabilistic
ideas and tools, employed in a very rich manner for the analysis
of information-theoretic questions.

Equally remarkable has been the mathematical
``traffic'' in the reverse direction.
Information-theoretic results and intuition 
have found applications in the study of
a rapidly growing
number of fundamental probabilistic phenomena.
In most cases, information theory not only informs our 
understanding of different aspects of stochastic behaviour,
but it also offers new avenues for exploring this
behaviour rigorously.

In 1958, H\'{a}jek used the properties of the
relative entropy to explore the absolute
continuity between~Gaussian measures~\cite{hajek:58a,hajek:58b}.
But the first major results established by the 
application of information-theoretic tools for 
the proof of purely probability-theoretic results
came in 1959, when 
Linnik suggested an information-theoretic proof of the 
central limit theorem~\cite{linnik:59},
and the following year, in 1960, when
R\'{e}nyi used information-theoretic methods to 
prove that finite-state Markov chains with all 
positive transitions convergence to equilibrium 
in relative entropy~\cite{renyi:61}. Then
in the 1970s and 80s, Csisz\'{a}r
used his `method of types'~\cite{csiszar:98}
to establish strong versions of a number of 
the standard results in
large deviations~\cite{csiszar:75,csiszar:84}.

\medskip

\noindent
{\bf Barron's influence. }
Over the past 40 years,
Andrew Barron has played a major role in promoting the information-theoretic
approach to probabilistic fundamentals. Here we briefly outline some of his
main contributions in this direction.

The Shannon-McMillan-Breiman theorem is probably the most
important foundational result of classical information 
theory, with important links to probability, statistics,
ergodic theory, dynamical systems, and beyond. Andrew Barron entered
the world of information theory in
1985 with his proof of the most general version of the
Shannon-McMillan-Breiman theorem for general ergodic
processes with densities~\cite{barron:1}.
In a very different direction, only a year later, he proved
what by now is the most well-known version of what has come
to be known as {\em the information-theoretic central limit 
theorem}~\cite{barron:clt}: He showed that the relative
entropy between the distribution of the standardised sum
of independent and identically distributed random variables,
and the Gaussian law, converges to zero if and only if it is
ever finite. This paper has been very influential, leading
to a number of subsequent breakthroughs 
including~\cite{artstein:04,johnson-barron:04,tulino:06,madiman:07}
and~\cite{gavalakis:clt}.
Then in 1991, Barron gave an information-theoretic proof
of the martingale convergence theorem
for nonnegative martingales,
and established a number of 
related convergence results for monotone
sequences of $\sigma$-algebras~\cite{barron:isit91}.
Building on earlier work by R\'{e}nyi~\cite{renyi:61}
and Fritz~\cite{fritz:73}, in 2000 Barron
established an elegant Markov chain convergence
theorem. He used information-theoretic arguments to show that,
for a reversible Markov chain (on a general state space)
with a unique invariant measure,
the relative entropy between the time-$n$ distribution
of the chain and its invariant measure, converges
to zero if it is ever finite~\cite{barron:isit00}.
And in 2007, in joint work with Madiman~\cite{madiman:07},
Barron developed a number of sharp inequalities
for the entropy and Fisher information of sums
of independent, continuous random variables, 
drawing interesting parallels with functional-analytic
and statistical considerations.

It is certainly fair to say that, 
at least between the mid-1980s
and the late 1990s, Andrew Barron was the main 
driving force of the information-theoretic
approach to probabilistic fundamentals.

\medskip

\noindent
{\bf More information in probability. } Over the past 30 years,
many more fascinating connections have been established
between information-theoretic ideas and core probabilistic
results. We only mention some of them very briefly;
more details can be found in the relevant historical
discussion in~\cite{gavalakis-LNM:23},
and in the earlier reviews by
Barron~\cite{barron:97,barron:isit00},
Csisz{\'a}r~\cite{csiszar:97},
and the text~\cite{johnson:book}.

The so-called {\em entropy method},
introduced 
Herbst~\cite{davies-simon:84} and
developed by 
Marton~\cite{marton:96} and
Ledoux~\cite{ledoux:96,ledoux:97,ledoux:book},
has been one of the key approaches to 
proving concentration of measure 
inequalities~\cite{boucheron:book,raginsky:book},
sometimes also in connection with ideas
from optimal transport~\cite{bobkov:99,villani:book}.
Poisson 
approximation~\cite{harremoes:01,konto-H-J:05,HJK:ITW08,konto-H-J:10}
and compound Poisson approximation~\cite{K-madiman:ITW04,Kcompound:10,jKm:13}
have been extensively studied via an
information-theoretic lens.
The profound relationship between information theory
and 
functional-analytic inequalities has a long history,
dating back to the work of Shannon~\cite{shannon:48},
Stam~\cite{stam:59} and Blachman~\cite{blachman:65}
on the entropy power inequality~\cite{dembo-cover-thomas:91}.
These include
re-interpretations of the
Gross' logarithmic 
Sobolev inequality~\cite{gross:75},
the Brascamp-Lieb 
inequality~\cite{carlen:09,courtade:16,courtade:21,anantharam:22},
and connections
with high-dimensional convex 
geometry~\cite{costa-cover:84,bobkov:12,madiman:17}.
The deep connections between the convergence of diffusions,
estimation, relative entropy, and Fisher information,
have led to the development of a rich web of results
known as `Bakry-{\'E}mery theory'~\cite{bakry-emery:85,bakry:book};
see also the relevant work by Brown~\cite{brown:86},
Barron~\cite{barron:97} and Guo et al.~\cite{guo:05}.
Motivated by fascinating developments in additive 
combinatorics and number theory,
a series of entropy bounds have been developed by,
among others, Ruzsa~\cite{Ruz09} and Tao~\cite{tao:10}.
More recent work in this direction
includes~\cite{KM:14,KM:16,fradelizi:arxiv24,%
gavalakis:24,gavalakis-doubling:arxiv}.
Elementary information-theoretic
results were applied to some classical
questions in probabilistic number theory
in~\cite{primes:ITW08,kontoyiannis:primes}.
Finally, {\em free entropy} plays a major role in the
noncommutative probability theory developed by
Voiculescu~\cite{voiculescu:I,voiculescu:II,voiculescu:III,voiculescu:IV}.

\section{de Finetti's representation theorem}
%%%%%%%%%%%%%%%%%%%%%%%%%%%%%%%%%%%%%%%%%%%%%%%%%%%%%%%%%%%%%%%%%%%%%%%%%%
\label{s:classical}

A random vector $X_1^n:=(X_1,\ldots, X_n)$ is {\em exchangeable}
if its distribution is invariant under permutations of the indices
$\{1,\ldots,n\}$. A process $\{X_n\;;\;n\geq 1\}$ is 
{\em exchangeable}
if $X_1^n$ is exchangeable for every $n \geq 1$. 
[Throughout, we use the notation $X_i^j$ to denote the vector 
of random variables $(X_i,\ldots,X_j), i\leq j$,
and the corresponding lower-case notation $x_i^j$ for
individual realisations $(x_i,\ldots,x_j)$ of $X_1^j$.]
De~Finetti's 
celebrated representation theorem, established
in the 1930s, states that a binary process 
is exchangeable if and only if it is a mixture of independent 
and identically distributed (i.i.d.) sequences.

\begin{theorem}[de Finetti~\cite{definetti:31,definetti:37}]
Let $\{X_n\}$ be an exchangeable process where each 
$X_n$ takes values in $\{0,1\}$. Then there is a unique 
Borel probability measure $\mu$ on $[0,1]$ such that, 
for every $n\geq 1$,
\begin{equation} 
\label{eq:def}
    \BBP(X_1^n = x_1^n) = \int_{[0,1]}{Q^n_{p}(x_1^n)\,d\mu(p)} 
	= \int_{[0,1]}{\prod_{i=1}^nQ_{p}(x_i)\,d\mu(p)},
\quad x_1^n\in\{0,1\}^n,
\end{equation}
where $Q_p(1) = 1- Q_p(0) = p$ is the 
probability mass function of the Bernoulli 
distribution with parameter $p$.
\end{theorem}

De Finetti's theorem holds much more generally.
For a measurable space $(S, \mathcal{S})$ we write $\MS$ 
for the 
space of probability measures on $S$,
and $\clF$ for the smallest 
$\sigma$-algebra that makes the maps 
$\{\pi_A\;;\;A \in \mathcal{S}\}$ measurable,
where each $\pi_A:\clM(S)\to[0,1]$ is 
defined by $P\mapsto \pi_A(P) := P(A)$.
For a probability measure $\mu$ on $(\MS,\clF)$ and $k\geq1$, 
we write $M_{k,\mu}$ for the mixture of i.i.d.\ measures on 
$(S^k,\clS^k)$, defined by:
\begin{equation}
\label{eq:mixture}
    M_{k,\mu}(A):= \int_{\MS}{Q^k(A)\,d\mu(Q)}, \quad A \in \mathcal{S}^k.
\end{equation}

In the case when $S$ is a finite set,
we often identify probability measures $P$ on $S$ 
with the corresponding probability mass functions (p.m.f.s),
so that $P(x)=P(\{x\})$, $x\in S$. Similarly
we identify $\clM(S)$ with
the corresponding simplex in $\RL^c$
consisting of all probability vectors
$(P(x)\;;x\in S)$, and we equip $\clM(S)$
with the Borel $\sigma$-algebra generated
by the open subsets of $\clM(S)$ in the
induced subspace topology.

The most general form of de Finetti's theorem 
is due to Hewitt and Savage:

\begin{theorem} [Hewitt and Savage~\cite{hewitt-savage:55}]
\label{thm:HS}
Let $S$ be a compact Hausdorff space 
equipped with its Baire $\sigma$-algebra $\mathcal{S}$.
If $\{X_n\}$ is an exchangeable process 
with values in $S$, 
then there exists a unique measure $\mu$ on 
the Baire $\sigma$-algebra of $\MS$ such that,
for each $k\geq 1$, the law $P_k$ of $X_1^k$
admits the representation:
    $$
    P_k = M_{k,\mu},
    $$
\end{theorem}

Recall that
the Baire $\sigma$-algebra $\clB$ of a topological space $B$
is the smallest $\sigma$-algebra that makes
all continuous functions $f:B\to\RL$ 
measurable~\cite{dudley:book,cohn:book}.
A key idea in the proof of Hewitt and Savage is 
the geometric interpretation of i.i.d.\ measures 
as the extreme points of the convex set of exchangeable measures,
so that any point in this convex set can be expressed as a mixture 
of sufficiently many extreme points. This idea also plays an important 
role in the {\em finite} de Finetti regime discussed in 
the following sections.   

There is extensive literature extending de Finetti's theorem 
in a number of different directions. Although in this paper we 
focus on {\em finite} de Finetti bounds and their 
connection with sampling bounds, 
we briefly mention some other interesting connections.
One exciting such connection is between de Finetti-style theorems
and what in the probability literature is known as {\em Gibbs'
conditioning principle}~\cite{dembo-zeitouni:book},
also referred to as the
{\em conditional limit theorem}~\cite{cover:book2}.
Diaconis and Freedman~\cite{diaconis-freedman:87}
showed that the first $k$ coordinates of an orthogonally 
invariant random vector in $\mathbb{R}^n$ are approximately 
mixtures of independent Gaussian random variables.
This is very similar, in spirit, both to 
de Finetti's theorem and to the conditional
limit theorem, which states the following:
Suppose $n$ is large and $k$ is fixed, 
and let $\hat{P}_{X_1^n}$ denote the empirical measure
induced by the i.i.d.\ vector $X_1^n\sim P^n$.
Then, conditional on $\hat{P}_{X_1^n}$
belonging to an atypical set $E$ of probability measures, 
the law of $X_1^k$ is approximately 
equal to the i.i.d.\ law $(P^*)^k$, 
where $P^*$ is the relative entropy-closest 
member of $E$ to $P$.  
Exploring this connection further, a finite 
de Finetti theorem was proved in~\cite{diaconis-freedman:88}
using a finite form of the conditional limit theorem.
Interestingly, one of the information-theoretic 
proofs~\cite{gavalakis-LNM:23}
recently developed (see Section~\ref{relativeentropysec}) 
is also based on the proof of the conditional limit theorem.

Another fascinating area is that of exchangeable random graphs, 
which is also connected to the well known array-version of 
de Finetti's theorem, known as the Aldous-Hoover theorem;
see. e.g.,~\cite{diaconis-janson:08} and references therein.
Extensions of de Finetti-style
theorems to mixtures of Markov chains were also established
by Diaconis and Freedman~\cite{diaconis-freedman:80}.
A conjecture on partial exchangeability made in that work 
was recently proven in~\cite{halberstam:24}.

Finally, we mention that there has been renewed interest
in exchangeability in statistics, in part motivated
by the success of conformal prediction and related 
methods~\cite{gammerman:98,saunders:99},
leading among other things to the notion of weighted exchangeability;
see, e.g., the recent work in~\cite{barber:23,tang:23}.

\section{Finite de Finetti bounds in total variation} 
%%%%%%%%%%%%%%%%%%%%%%%%%%%%%%%%%%%%%%%%%%%%%%%%%%%%%%%%%%%%%%%%%%%%%%%%%%
\label{s:dfTV}
 %is beautiful and clear. State and explain the Diaconis and the
 %Diaconis-Freedman bounds . Recall the geometric interpretation via
 %convexity, and the connection with sampling bounds.}

De Finetti's theorem offers an explicit characterisation of 
exchangeable processes, which is general, natural, and useful.
Historically, it has also been viewed as a powerful justification
of the subjective probability point of view 
in Bayesian statistics~\cite{diaconis:77,bayarri:04}.
In this context, it is interpreted as stating that,
an exchangeable binary sequence, for example,
can equivalently be viewed as the realisation
of an i.i.d.\ Bernoulli sequence, conditional on the
value of the Bernoulli parameter, which is distributed
according to a unique prior distribution.

In terms of applications, it is natural to ask whether
de Finetti's theorem also holds for finite exchangeable
sequences. The answer is ``yes and no''.
Strictly speaking, the exact representation of the 
distribution of a finite exchangeable vector as a mixture
of product distributions does {\em not} hold in general,
but it does hold approximately.
Consider, e.g., a pair $(X_1,X_2)$ of binary
random variables with,
\begin{align} \label{cexample}
\BBP(X_1=0, X_2 = 1) = \BBP(X_1 = 1, X_2 = 0) = \frac{1}{2}.
\end{align}
The random vector $(X_1,X_2)$ is clearly exchangeable, but if 
a representation like~\eqref{eq:def} were true for some 
probability measure $\mu$ on $[0,1]$, then we would have,
$$
\int_{[0,1]}p^2\,d\mu(p) = \int_{[0,1]} (1-p)^2\,d\mu(p) = 0,
$$
which implies that $\mu(\{0\}) = \mu(\{1\}) = 1$, a contradiction. 

The example~\eqref{cexample} was given by Diaconis~\cite{diaconis:77}. 
In that work, the set of exchangeable binary measures for which a de 
Finetti-style representation fails was interpreted in a geometric way,
and it was observed that the volume of the region representing those measures 
decreases, in the following sense: If $X_1^k$ are the
first $k$ coordinates of a longer exchangeable binary sequence $X_1^n$, 
then for $n$ significantly larger than $k$ the distribution of $X_1^k$
is close to a product distribution. More specifically, 
it was shown that for each $k\leq n$ there exists a 
mixing measure $\mu_n$, depending on $n$ but not on $k$, 
such that the distribution $P_k$ of $X_1^k$
satisfies,
\begin{equation} \label{eq:diaconisfirst}
\tv{P_k}{M_{k,\mu_n} } \leq \frac{C_k}{n},
\end{equation}
where $C_k$ is a constant depending only on $k$. 
Here the total variation distance (TV) between two 
measures $\mu, \nu \in \MS$ is defined as:
$$
\tv{\mu}{\nu} := 2\sup_{A \in \mathcal{S}}{|\mu(A) - \nu(A)|}.
$$
Not coincidentally, as we will see below,
the hypergeometric probabilities appear in the 
geometric proof of~(\ref{eq:diaconisfirst}) as the extreme points of
the convex set of exchangeable measures embedded in $\mathbb{R}^k$.
Results of the form~\eqref{eq:diaconisfirst} are referred to as 
{\em finite de Finetti theorems}, and they typically state that,
if $X_1^n$ is exchangeable and $k$ is small compared to $n$, 
then the distribution of $X_1^k$ is in some sense close
that of of an i.i.d.\ mixture. 

The binary assumption was removed and the the sharpest rates were 
obtained in Diaconis and Freedman~\cite{diaconis-freedman:80b},
for exchangeable random vectors with values in an arbitrary measurable
space:

\begin{theorem}[Diaconis and Freedman~\cite{diaconis-freedman:80b}]
\label{dfinfinite} 
Let $X_1^n$ be an exchangeable random vector with values in
a measurable space 
$(S,\mathcal{S})$.
Then there exists 
a probability measure $\mu_n$ 
on $(\MS,\clF)$ such that, for every $k \leq n$,
the distribution $P_k$ of $X_1^k$ satisfies,
\begin{equation} \label{eq:dfindep}
\tv{P_k}{M_{k,\mu_n}} \leq \frac{k(k-1)}{2n},
\end{equation}
where $M_{k,\mu_n}$ is the mixture of product distributions
in~{\em (\ref{eq:mixture})}.
\end{theorem}

Unlike the geometric proof of~(\ref{eq:diaconisfirst}), the proof
of~(\ref{eq:dfindep}) was based on the 
elegant connection of finite de Finetti representations
with bounds on the difference of 
sampling without and with replacement, described next.
We will make use of a similar argument 
in Section~\ref{newboundsec}, 
adapted for relative entropy.

\medskip

\noindent
{\bf Empirical measures and types. }
Let $\hat{P}_{X_1^n}$ denote the {\em empirical measure} 
induced on $(S,\mathcal{S})$ by the 
random vector, namely,
$$
\hat{P}_{X_1^n} := \frac{1}{n}\sum_{i=1}^n\delta_{X_i},
$$
where $\delta_x$ denotes the Dirac measure that places a unit
mass at $x\in S$. 
Similarly, let $\hat{P}_{x_1^n}$ denote the empirical measure 
induced by
a fixed string $x_1^n\in S^n$.
We refer to $\hat{P}_{x_1^n}$ as the
{\em type} of $x_1^n$.
If $Q$ is a p.m.f.\ on a finite set $S_0$,
or a probability measure supported on a finite
subset $S_0\subset S$, we call $Q$ an $n$-{\em type}
if $nQ(x)$ is an integer for each $x\in S_0$.

\medskip

\noindent
{\bf Exchangeability and sampling. }
The key observation here is the following: Let $X_1^n$ be an
exchangeable random vector. Then,
conditional on $\hat{P}_{X_1^n}=Q$,
the distribution of
$X_1^n$ is uniform on the set of sequences 
$x_1^n$ with the same type $Q$. Moreover, for any $k\leq n$,
the distribution $P_k$ of $X_1^k$ is the
distribution of sampling without replacement from an urn 
containing the balls $x_1,\ldots,x_n$, where $x_1^n$ 
is any string with $\hat{P}_{x_1^n}=Q$. Therefore,
letting $\mu_n$ denote the law of $\hat{P}_{X_1^n}$
on $(\clM(S),\clF)$, we have,
\begin{equation}
P_k(A)=\int h(Q,n,k;A)\,d\mu_n(Q),\quad A\in \clS^k,
\label{eq:Pkintegral}
\end{equation}
where $h(Q,n,k;\cdot)$ denotes the 
{\em multivariate hypergeometric}
law of drawing $k$ balls {\em without} 
replacement from
an urn containing the balls $x_1$, $x_2$, \ldots, $x_n$,
where $x_1^n$ is any string in $S^n$ with type $\hat{P}_{x_1^n}=Q$.
Similarly, the mixture $M_{k,\mu_n}$ with
respect to the same mixing measure $\mu_n$ can be
written,
\begin{equation}
M_{k,\mu_n}(A)=\int b(Q,n,k;A)\,d\mu_n(Q),\quad A\in \clS^k,
\label{eq:Mkintegral}
\end{equation}
where $b(Q,n,k;\cdot)=Q^k(\cdot)$ denotes the 
{\em multinomial}
law of drawing $k$ balls {\em with}
replacement from an urn containing 
$x_1$, $x_2$, \ldots, $x_n$, for a string
$x_1^n\in S^n$ with type $\hat{P}_{x_1^n}=Q$.
In view of the expressions~(\ref{eq:Pkintegral})
and~(\ref{eq:Mkintegral}),
comparing $P_k$ with $M_{k,\mu_n}$ reduces
to comparing the hypergeometric and multinomial 
distributions, as described in more detail 
in Section~\ref{samplingsec}. 

Note that, as mentioned in~\cite{diaconis-freedman:80b},
the rigorous justification of the representations~(\ref{eq:Pkintegral})
and~(\ref{eq:Mkintegral}) follows from the measurability
of $\hat{P}_{X_1^n}$ discussed in the Appendix,
and the obvious measurability of the
p.m.f.s $h(Q,n,k;\cdot)$ and $b(Q,n,k;\cdot)$ as
functions of $Q=\hat{P}_{X_1^n}$.

\medskip

The above argument strongly indicates
that the ``natural" mixing measure 
to consider for finite de Finetti theorems 
is the law of 
the empirical measure $\hat{P}_{X_1^n}$ induced by $X_1^n$.

Diaconis and Freedman also showed that the  $O(k^2/n)$ rate
of Theorem~\ref{dfinfinite} can
be improved to $O(k/n)$ when the space 
$S$ is finite:

\begin{theorem}[Diaconis and Freedman~\cite{diaconis-freedman:80b}]
\label{dffinite}
Let $X_1^n$ be an exchangeable random vector with values in a finite $S$
of cardinality $c$.
Then there exists a Borel probability measure $\mu_n$ on 
$\MS$ such that, for every $k\leq n$,
the distribution $P_k$ of $X_1^k$ satisfies,
\begin{equation}
\label{eq:dffinite}
    \tv{P_k}{M_{k,\mu_n}} \leq \frac{2ck}{n},
\end{equation}
where $M_{k,\mu_n}$ is the mixture of product distributions
in~{\em (\ref{eq:mixture})}.
\end{theorem}

The rates in both bounds~\eqref{eq:dfindep} 
and~\eqref{eq:dffinite} 
were shown in~\cite{diaconis-freedman:80b} to be tight. We explain 
the tightness of~\eqref{eq:dfindep} at the end of
Section~\ref{samplingsec} below, as it follows from the tight 
approximation of sampling without and with replacement. 

\section{Sampling bounds in total variation} 
%%%%%%%%%%%%%%%%%%%%%%%%%%%%%%%%%%%%%%%%%%%%%%%%%%%%%%%%%%%%%%%%%%%%%%%%%%%%%
\label{samplingsec}

Consider an urn containing $n$ balls, each ball having one 
of $c\geq 2$ different colours, and suppose we draw out $k\leq n$ of them. 
Beyond the motivation offered above in connection
with de Finetti-style representations, 
comparing the distributions of sampling with and without replacement 
is a fundamental problem with a long 
history in probability and statistics;
see, e.g.,~\cite{rao:66, thompson:12}.
Intuitively, 
if the number $n$ of balls is large compared to the number $k$ 
of draws, there should only be a negligible difference in the results
of sampling with and without replacement.  
Write $\ellp = (\ell_1,\ell_2, \ldots, \ell_c)$ for the vector 
representing the number of balls of each colour in the urn, 
so that there are $\ell_j$ balls of colour $j$, $1\leq j\leq c$ 
and $\ell_1+\ell_2+\cdots+\ell_c = n$. 
Let $\spro=(s_1,s_2,\ldots,s_c)$ denote the vector 
of the numbers of balls of each colour drawn, 
so that $s_1+s_2+\cdots+s_c = k$.

When sampling without replacement, the probability that the
colours of the $k$ balls drawn are given by $\spro$ 
is given by the multivariate hypergeometric 
p.m.f.,
\begin{equation}
H(\ellp,n,k;\spro) 
:= \frac{ \prod_{j=1}^c \binom{\ell_i}{s_j}}{\binom{n}{k}}
\label{eq:MVhyp}
\end{equation}
for all $\spro$ with
$0\leq s_j\leq\ell_j$ for all $j$, and
$s_1+\cdots+s_c=k$. 
On the other hand, the corresponding p.m.f.\ $B(\ellp,n, k; \spro)$
of sampling with replacement, is the multinomial,
\begin{equation} 
\label{eq:multin}
B(\ellp, n, k; \spro) 
:= \binom{k}{s_1,\ldots,s_c} \prod_{j=1}^c \left( \frac{\ell_j}{n} 
\right)^{s_j},
\end{equation}
for all $\spro$ with $s_j\geq 0$ and
$s_1+\cdots+s_c=k$, 
where $\binom{k}{s_1,\ldots,s_c} = \frac{k!}{\prod_{j=1}^cs_j!}$ 
is the multinomial coefficient.

Note that the p.m.f.s $H$ and $B$ 
in~(\ref{eq:MVhyp}) and~(\ref{eq:multin}) involve only 
the numbers of balls of each colour that are drawn,
whereas the corresponding distributions $h$ and $b$
defined in the previous section
are over the entire sequence of colours drawn
from the urn. 
Of course the two are simply related:
Suppose the composition of the urn is described
by the vector $\ellp$ or,
equivalently, by the $n$-type
$Q^{(\ellps)}$ defined by
$Q^{(\ellps)}(j):=\ell_j/n$, for each colour
$j=1,2,\ldots,c$.
For the sake of simplicity (and without
loss of generality), take the set of colours $S$
to be $S=\{1,2,\ldots,c\}$.
Then, by the definitions of $H,B,h$ and $b$,
for any $x_1^k\in S^k$,
we have,
\begin{align} 
h(Q^{(\ellps)},n,k;x_1^k)
&= 
	\binom{k}{s_1,\ldots,s_c}^{-1}
	H(\ellp,n,k;\spro) 
	\label{eq:relabelH}\\
b(Q^{(\ellps)},n,k;x_1^k) 
&= 
	\binom{k}{s_1,\ldots,s_c}^{-1}
	B(\ellp,n,k;\spro),
\label{eq:relabelB}
\end{align}
where $\spro$ is the composition of $x_1^k$,
i.e., each $s_j$ is the number of occurrences
of $j$ in $x_1^k$, $1\leq j\leq c$.

Diaconis and Freedman established the following bound 
between $h$ and $b$, and used it to prove
Theorem~\ref{dffinite} using the
connection between exchangeability and sampling explained in the 
previous section.

\begin{theorem} 
[Diaconis and Freedman~\cite{diaconis-freedman:80b}]
\label{dfsampling} 
Let $h(Q,n,k;\cdot)$ and $b(Q,n,k;\cdot)$ denote 
p.m.f.s of sampling $k$ balls without and with replacement,
respectively, from an urn containing $n$ balls of $c$ different 
colours, where the $n$-type $Q$ describes the composition
of the balls in the urn. Then:
$$
\tv{h(Q,n,k;\cdot)}{b(Q,n,k;\cdot)} \leq \frac{2ck}{n}.
$$
\end{theorem}

However, if one wants a bound that is independent of the number of colours, 
then one has to pay a factor $k$ in the bound:

\begin{theorem} 
[Freedman~\cite{freedman:77}]
\label{thm:freedmancn}
In the notation of Theorem~\ref{dfsampling},
suppose $k$ balls are drawn from an urn containing
$n$ balls of $n$ different colours, so that $c=n$ and
$Q=Q_U$ with $Q_U(j)=1/n$ for each $j=1,2,\ldots,n$.
Then:
\begin{equation} \label{eq:freedman}
    2\bigl(1 - e^{-\frac{k(k-1)}{2n}}\bigr) 
\leq\tv{h(Q_U,n,k;\cdot)}{b(Q_U,n,k;\cdot)} \leq \frac{k(k-1)}{n}.
\end{equation}
\end{theorem}

The proof of~\eqref{eq:freedman} is based on considering
the set,
\begin{equation} \label{birthdayset}
    B = \{x_1^k : x_i=x_j \text{ for some } 1\leq i<j\leq k\},
\end{equation}
and noting that $h(Q_U,n,k;x_1^k) = 0$ for all $x_1^k\in B$,
which implies,
\begin{equation}
\frac{1}{2}\tv{h(Q_U,n,k;\cdot)}{b(Q_U,n,k;\cdot)} 
= 1 - \frac{n!}{(n-k)!n^k}.
\label{eq:exact}
\end{equation}

Getting back to de Finetti's theorem, Diaconis 
and Freedman~\cite[Proposition~31]{diaconis-freedman:80b} 
show the inequality,
\begin{equation*}
    \tv{h(Q_U,n,k;\cdot)}{M_{k,\mu}(\cdot)} 
	\geq \tv{h(Q_U,n,k;\cdot)}{b(Q_U,n,k;\cdot)}
\end{equation*}
for any mixing measure $\mu$. Therefore, 
since $h(Q_U,n,k;\cdot)$ is the distribution of $X_1^k$
when $X_1^n$ is the (exchangeable) vector obtained by 
a random permutation of $S=\{1,\ldots,n\}$,
the sharpness of~\eqref{eq:dfindep} 
follows from the 
sharpness of the upper bound in~\eqref{eq:freedman}.

\section{Sampling bounds in relative entropy}  
%%%%%%%%%%%%%%%%%%%%%%%%%%%%%%%%%%%%%%%%%%%%%%%%%%%%%%%%%%%%%%%%%%%%%%%%%%%%%

For $\mu, \nu \in \MS$, the relative entropy 
between $\mu$ and $\nu$ is defined as,
\begin{equation*}
    D(\mu\|\nu) := \int_{S}{\frac{d\mu}{d\nu}\log{\frac{d\mu}{d\nu}}\,d\nu},
	 \quad \text{if } \mu \ll \nu,
\end{equation*}
and $D(\mu\|\nu) = \infty$ otherwise, where $\frac{d\mu}{d\nu}$ stands 
for the Radon-Nikod{\'y}m derivative of $\mu$ with respect to $\nu$,
and where $\log$ denotes the natural logarithm throughout.
In particular, if $S$ is discrete and $P,P'$ are the p.m.f.s
corresponding to $\mu,\nu$, then,
$$D(\mu\|\nu)=D(P\|P')=\sum_{x\in S:P(x)>0}P(x)\log\frac{P(x)}{P'(x)}.$$
In view of Pinsker's inequality~\cite{csiszar:67, kullback:67},
\begin{equation*} % \label{pinsker}
    \tv{\mu}{\nu} \leq \big[2D(\mu\|\nu)\big]^{1/2},
\end{equation*}
relative entropy is considered 
a stronger notion of ``distance" than total variation.
It is $0$ if and only if $\mu = \nu$,
and it is locally quadratic around $\mu=\nu$~\cite{pardo:03}. 
Moreover, although not a proper metric, relative entropy is 
often thought of 
as a notion of distance between the two measures $\mu$ and $\nu$, 
justified in part by important results in probability and 
statistics~\cite{csiszar-shields:04}.

The difference between sampling with and without replacement 
has also been studied in terms of relative entropy: 

\begin{theorem}[Stam~\cite{stam:78}]
\label{th:stam}
Let $H(\ellp,n,k;\cdot)$ and $B(\ellp,n,k;\cdot)$ denote 
p.m.f.s of sampling without and with replacement 
from an urn with balls of $c$ colours,
as in~{\em (\ref{eq:MVhyp})} and~{\em (\ref{eq:multin})}.
Then, for any $\ellp$ and any $k\leq n$:
\begin{equation} 
\label{eq:stam}
D\big(H(\ellp,n,k;\cdot)\|B(\ellp,n,k;\cdot)\big) 
\leq \frac{(c-1)k(k-1)}{2(n-1)(n-k+1)}.
\end{equation}
\end{theorem}

Based on~(\ref{eq:relabelH}) and~(\ref{eq:relabelB}),
Stam~\cite[Theorem~2.2]{stam:78} observed that,
\begin{equation} \label{stamequivalence}
    D\bigl(h(Q^{(\ellps)},n,k;\cdot)\|b(Q^{(\ellps)},n,k;\cdot)\bigr) 
	= D\big(H(\ellp,n,k;\cdot)\|B(\ellp,n,k;\cdot)\big),
\end{equation}
so that we also have:
\begin{equation} \label{eq:stamlast}
    D\bigl(h(Q^{(\ellps)},n,k;\cdot)\|b(Q^{(\ellps)},n,k;\cdot)\bigr) 
\leq \frac{(c-1)k(k-1)}{2(n-1)(n-k+1)}.
\end{equation}
Moreover, Stam established a closely matching lower bound,
showing that the $O(k^2/n^2)$ upper bound in Theorem~\ref{th:stam}
is of optimal order in terms of its dependence on $k$ and $n$,
but in general it can be improved:
Harremo\"{e}s and Mat\'{u}\v{s}~\cite[Theorem~4.5]{harremoes:20}
showed that:
\begin{equation}
\label{eq:harmat}
 D\big(H(\ellp,n,k;\cdot)\|B(\ellp,n,k;\cdot)\big) 
\leq (c-1) \left(  \log \left( \frac{n-1}{n-k} \right) 
- \frac{k}{n} + \frac{1}{n-k+1} \right).
\end{equation}

In the special case $c=n$ as in Theorem~\ref{thm:freedmancn},
an exact expression is derived in~\cite[eq.~(29)]{gavalakis-olly:arxiv},
which is interesting to compare with~(\ref{eq:exact}) above:
$$D\big(H(\ellp,n,k;\cdot)\|B(\ellp,n,k;\cdot)\big) 
=\log\Big(\frac{n^k(n-k)!}{n!}\Big).$$

Note that all the bounds
in~\eqref{eq:stam},~(\ref{eq:stamlast}) and~\eqref{eq:harmat} hold 
uniformly in $\ellp$. Sharper bounds can be obtained
if we allow dependence on $\ellp$. Indeed, a bound which
is often sharper was recently given 
in~\cite[Theorem~1.1]{gavalakis-olly:arxiv}:
For any $\ellp$ and all $1\leq k \leq n/2$:
\begin{align}
D\big(H(\ellp,n,k;\cdot)\|B(\ellp,n,k;\cdot)\big) & \leq 
\frac{c-1}{2} \Big( \log \Big( \frac{n}{n-k} \Big)
- \frac{k}{n-1} \Big) 
	\nonumber\\
&   +  \frac{k(2n+1)}{12n(n-1)(n-k)} \sum_{i=j}^c \frac{n}{\ell_j}
+ \frac{1}{360} \Big( \frac{1}{(n-k)^3}  - \frac{1}{n^3} \Big) 
\sum_{j=1}^c \frac{n^3}{\ell_j^3}.\qquad 
 \label{eq:olly}
\end{align}
See~\cite{gavalakis-olly:arxiv} for some
detailed comparisons 
between~\eqref{eq:stam},~(\ref{eq:stamlast}),~\eqref{eq:harmat}
and~(\ref{eq:olly}).

Finally, we emphasise that all the bounds in this section depend 
(in fact, linearly) on the number of colours $c$.

\section{Finite de Finetti bounds in relative entropy} 
%%%%%%%%%%%%%%%%%%%%%%%%%%%%%%%%%%%%%%%%%%%%%%%%%%%%%%%%%%%%%%%%%%%%%%%%%%%%%
\label{relativeentropysec}

A number of finite de Finetti bounds in relative entropy have
recently been established  
in~\cite{gavalakis:21,gavalakis-LNM:23,gavalakis-berta:24,%
gavalakis-olly:arxiv,yu:24}. 
Let $X_1^n$ be an exchangeable random vector with 
values in some space $(S,\clS)$, let $P_k$ denote
the law of $X_1^k$ for $k\leq n$, and 
for any probability measure $\mu$ on
$(\clM(S),\clF)$, recall
the definition of the mixture of product distributions
$M_{k,\mu}$ in~(\ref{eq:mixture}).

In~\cite{gavalakis:21}, it was shown that,
if $S=\{0,1\}$, then 
there is
a mixing measure $\mu_n$ such that, for $k\leq n$,
\begin{equation} \label{gkfirsteq}
D(P_{k}\|M_{k,\mu_n}) \leq \frac{5k^2\log n}{n-k}.
\end{equation}
Then in~\cite{gavalakis-LNM:23} it was shown that,
if $S$ is a finite set, then
there is a mixing measure $\mu_n$ such that
a weaker bound of the following form holds
for all $k$ sufficiently smaller than $n$:
\begin{equation} \label{gksecondeq}
D(P_{k}\| M_{k,\mu_n})=
O\left(\Big(\frac{k}{\sqrt{n}}\Big)^{1/2}\log{\frac{n}{k}}\right).
\end{equation}
The proofs of both of these results are information-theoretic,
and in both cases the mixing measure $\mu_n$ is the law of the
empirical measure $\hat{P}_{X_1^n}$.
The bound~\eqref{gkfirsteq} was proved via conditional entropy 
estimates, while the proof of~(\ref{gksecondeq}) explored the 
connection between exchangeability and the Gibbs conditioning principle.

In the more general case when $S$ is an arbitrary 
discrete (finite or countably infinite)
set, it was shown in~\cite{gavalakis-berta:24} 
that (a different) mixing measure
$\mu^*_n$ exists, such that the following sharper bound holds
for all $k<n$:
\begin{equation} \label{gkbeq}
D(P_{k}\|M_{k,\mu^*_n}) \leq \frac{k(k-1)}{2(n-k-1)}H(X_1).
\end{equation}
Note that this meaningful as long as $H(X_1)$ is finite, and
that it gives potentially much sharper estimates when
$H(X_1)$ is small.
De Finetti-type bounds for random variables with values 
in abstract spaces $(S,\clS)$ were also derived in~\cite{gavalakis-berta:24}.
The proof of~\eqref{gkbeq} was based on an argument that originated
in the quantum information theory literature. 

The derivations of~\eqref{gkfirsteq}--\eqref{gkbeq} 
employed purely information-theoretic
ideas and techniques, but the actual bounds
are of sub-optimal rate. A sharp rate 
was more recently obtained in~\cite{gavalakis-olly:arxiv},
using Stam's sampling bound 
in Theorem~\ref{th:stam} combined with the convexity
of relative entropy:

\begin{theorem} [Johnson, Gavalakis and 
	Kontoyiannis~\cite{gavalakis-olly:arxiv}]
\label{thm:withOlly}
Let $X_1^n$ be an exchangeable random vector with values 
in a finite set $S$ of cardinality $c$.
Then there is a Borel probability measure $\mu_n$ on $\clM(S)$
such that, for every $k \leq n$, the distribution $P_k$
of $X_1^k$ satisfies:
    \begin{equation} \label{eq:ollydf}
        D(P_k\|M_{k,\mu_n}) \leq \frac{(c-1)k(k-1)}{2(n-1)(n-k+1)}.
    \end{equation}
\end{theorem}

Theorem~\ref{thm:withOlly} gives
a bound with rate $O(k^2/n^2)$.
In view of Pinsker's inequality and the lower bound
in total variation obtained by Diaconis 
and Freedman~\cite{diaconis-freedman:80b},
this bound is of optimal order in its dependence 
on $k$ and $n$.

Even more recently,
Song, Attiah and Yu~\cite{yu:24} used an adaptation 
of Freedman's argument~\cite{freedman:77}, 
again based on considering the set~\eqref{birthdayset}, 
to establish a finite de Finetti theorem
with a relative entropy bound that is weaker
when $S$ is finite,
but which is independent 
of the alphabet size:

\begin{theorem} [Song, Attiah and Yu~\cite{yu:24}]
\label{thm:song}
Let $X_1^n$ be an exchangeable random vector with values
in a discrete (finite or countably infinite) set $S$.
Then there exists a probability measure $\mu_n$ on $(\MS,\clF)$, such that, 
for all $k<n$, the distribution $P_k$ of $X_1^k$ satisfies,
    \begin{equation} \label{yu:eq}
D(P_k\|M_{k,\mu_n}) \leq 
\log{\Bigl(\frac{n^k(n-k!)}{n!}\Bigr)} 
\leq -\log{\Bigl(1-\frac{k(k-1)}{2n}}\Bigr),
\end{equation}
where the second inequality holds as long as $k(k-1)<2n$.
\end{theorem}

In the same work, the bound of Theorem~\ref{thm:song}
was shown to be tight.
Recall the notation
$h(Q,n,k;\cdot)$ for the law of sampling without
replacement as in Section~\ref{samplingsec}.
Recall also the exact expression in~(\ref{eq:exact}) above.

\begin{theorem} [Song, Attiah and Yu~\cite{yu:24}]
\label{yuconverse}
Let $h(Q_U,n,k;\cdot)$ denote the law of sampling
$k$ balls without replacement from an urn
containing $n$ balls of $n$ different
colours, where $Q_U(j)=1/n$ for $j\in S=\{1,\ldots,n\}$.
Then, for any mixing measure $\mu$ on $\clM(S)$
and any $k<n$:
    \begin{equation}
        D\big(h(Q_U,n,k;\cdot)\|M_{k,\mu}(\cdot)\big) 
	\geq \log{\Bigl(\frac{n^k(n-k)!}{n!}\Bigr)}.
    \end{equation}
\end{theorem}

Letting as before $X_1^n$ denote the exchangeable random 
vector obtained as a random permutation
of the set $S=\{1,\ldots,n\}$, and noting that
$h(Q_U,n,k;\cdot)$ is then the same
as the distribution $P_k$ of $X_1^k$, Theorem~\ref{yuconverse}
shows that the bound~(\ref{yu:eq}) is indeed tight.

It is interesting to note that the 
finite de Finetti bound in~(\ref{yu:eq})
is used in~\cite{yu:24} in the proof of a strong achievability 
result for a certain source coding scenario, where an encoder 
transmits a $k$-letter information sequence to a randomly 
activated subset of $k$ out of $n$ possible users. 

In Section~\ref{newboundsec} we give two new finite
de Finetti bounds in relative entropy, that
are essentially tight, for exchangeable random
vectors with values in arbitrary measurable spaces.
Theorem~\ref{newtheorem} is proved by combining
Stam's sampling bound in Theorem~\ref{th:stam}
with the representations of $P_k$ and $M_{k,\mu_n}$
in terms of sampling distributions in~(\ref{eq:Pkintegral})
and~(\ref{eq:Mkintegral}).
The proof of Theorem~\ref{newtheorem2} is 
a generalization of the proof of Theorem~\ref{thm:song}
in~\cite{yu:24},
combined with the classical representation of relative entropy
as a supremum over finite partitions.

\section{New finite de Finetti bounds on abstract spaces} 
%%%%%%%%%%%%%%%%%%%%%%%%%%%%%%%%%%%%%%%%%%%%%%%%%%%%%%%%%%%%%%%%%%%%%%%%%%
\label{newboundsec}

Recall from Section~\ref{s:classical}
the definition of the $\sigma$-algebra $\clF$
associated with the space $\clM(S)$ of probability measures
on an arbitrary measurable space $(S,\clS)$,
and the definition of the mixture of i.i.d.\ measures $M_{n,\mu}$
in~(\ref{eq:mixture}).

\begin{theorem} \label{newtheorem}
Let $X_1^n$ be an exchangeable random vector with
values in a measurable space $(S,\clS)$.
Then 
there is a probability measure $\mu_n$ on 
$(\clM(S),\clF)$ such that, for each $1\leq k\leq n$,
the distribution $P_k$ of $X_1^k$ satisfies:
$$
D(P_k\|M_{k,\mu_n}) \leq \frac{k(k-1)}{2(n-k+1)}.
$$
\end{theorem}

\begin{theorem}
\label{newtheorem2}
Under the same assumptions
as Theorem~\ref{newtheorem},
there is a probability measure $\mu_n$ on 
$(\clM(S),\clF)$ such that, for each $1\leq k\leq n$,
$$
D(P_k\|M_{k,\mu_n}) \leq 
\log{\Bigl(\frac{n^k(n-k!)}{n!}\Bigr)} 
\leq -\log{\Bigl(1-\frac{k(k-1)}{2n}}\Bigr),$$
where the second inequality holds as long as $k(k-1)<2n$.
\end{theorem}

\noindent
{\bf Remarks. }
Before giving the proofs of Theorems~\ref{newtheorem}
and~\ref{newtheorem2}, some remarks 
are in order.
\begin{enumerate}
\item
In both theorems, the mixing measure
$\mu_n$ is the law of the empirical measure $\hat{P}_{X_1^n}$
on $\clM(S)$.
\item
Also both theorems give bounds with the same asymptotic behaviour
$\sim k^2/2n$.
The bound in 
Theorem~\ref{newtheorem2} is slightly stronger 
than that in Theorem~\ref{newtheorem},
and in fact,
in view of the lower bound in Theorem~\ref{yuconverse},
it is tight. On the other hand, the proof of Theorem~\ref{newtheorem}
is very short and quite satisfying in that
it is based on a very direct connection with Stam's~\cite{stam:59} sampling
bound~(\ref{eq:stamlast}).
\item 
As discussed by Diaconis and Freedman
in~\cite{diaconis-freedman:80b}, we note that it is
curious that we have finite de Finetti bounds, both in 
relative entropy and in total variation, under no assumptions
whatsoever on the underlying space $(S,\clS)$, whereas for the seemingly
weaker infinite-dimensional representation of the classical
de Finetti theorem in Theorem~\ref{thm:HS} more structure is required.
\item
The choice of the $\sigma$-algebra $\mathcal{F}$ 
on $\clM(S)$ is essential for the 
measurability of the empirical measure $\hat{P}_{X_1^n}$,
which is needed to define $\mu_n$ for both theorems.
Indeed, for stronger 
$\sigma$-algebras, the measurability of Dirac measures may fail.
For example, if $S$ is a Polish space and we equip $\mathcal{M}(S)$ 
with the Borel $\sigma$-algebra with respect to the $\tau$-topology, 
the Dirac measures are not measurable any more, and therefore 
neither is the empirical measure $\hat{P}_{X_1^n}$;
see, e.g., the relevant discussion
in~\cite[Section~6.2]{dembo-zeitouni:book}.
\item
As also mentioned in~\cite{gavalakis-berta:24}, the dependence 
of de Finetti-relative entropy upper bounds on the alphabet size is 
of interest in applications. For example, it is related to the running time 
of approximation schemes for the minimisation of polynomials
of fixed degree over the simplex~\cite{deklerk:06,deklerk:15}.
\end{enumerate}

\begin{proof}[Proof of Theorem~\ref{newtheorem}]
As it turns out, we have already done almost all of the
work for this. Using the definition
of $M_{k,\mu_n}$, the sampling representations
in~(\ref{eq:Pkintegral}) and~(\ref{eq:Mkintegral}),
and Stam's bound~(\ref{eq:stamlast}),
\begin{align*}
D(P_k\|M_{k,\mu_n})
&=
	D\left(\int h(Q,n,k;\cdot)\,d\mu_n(Q)
	\right\|\left.
	\int b(Q,n,k;\cdot)\,d\mu_n(Q)\right)\\
&\leq
	\int D\big(h(Q,n,k;\cdot)\| b(Q,n,k;\cdot)\big)\,d\mu_n(Q)\\
&\leq
	\frac{k(k-1)}{2(n-k+1)},
\end{align*}
where the first inequality follows from the joint convexity
of relative entropy in its two arguments~\cite{pinsker:book},
and the second inequality follows from the fact that, for any
specific $n$-type $Q$, the urn with composition described 
by $Q$ contains balls of at most $n$ different colours,
so we can take $c\leq n$ in~(\ref{eq:stamlast}).
\end{proof}

\begin{proof}[Proof of Theorem~\ref{newtheorem2}]
We begin by generalising a lower bound 
in~\cite{yu:24} 
based on an argument by Freedman~\cite{freedman:77}.
Take $W_1,W_2,\ldots,W_k$ to
be i.i.d.\ and uniformly distributed in $\{1,2,\ldots,n\}$,
independent of $X_1^n$.
For any $A\in\clS$ and any $n$-type $Q\in\clM(S)$,
from the definitions
it is easy to see that,
$$
\BBP\big((X_{W_1},X_{W_2},\ldots,X_{W_k})\in A
\big|\hat{P}_{X_1^n}=Q\big)=Q^k(A).$$
Therefore, the mixture measure $M_{k,\mu_n}$ can be written,
\begin{align*}
M_{k,\mu_n}(A)
&=
	\int 
	\BBP\big((X_{W_1},X_{W_2},\ldots,X_{W_k})\in A
	\big|\hat{P}_{X_1^n}=Q\big)\,d\mu_n(Q)\\
&=
	\BBP\big((X_{W_1},X_{W_2},\ldots,X_{W_k})\in A\big)\\
&=
	\sum_{w_1^k\in\{1,\ldots,n\}^k}
	\frac{1}{n^k}\,
	\BBP\big((X_{w_1},X_{w_2},\ldots,X_{w_k})\in A\big).
\end{align*}
Summing instead over all the index vectors $w_1^k$ 
in the subset $D_k$ of $\{1,\ldots,n\}^k$ 
that consists of all those $w_1^k$ that have $k$ distinct
elements, and using the exchangeability of $X_1^n$,
\begin{align*}
M_{k,\mu_n}(A)
&\geq
	\sum_{w_1^k\in D_k}
	\frac{1}{n^k}\,
	\BBP\big((X_{w_1},X_{w_2},\ldots,X_{w_k})\in A\big)\\
&=
	\sum_{w_1^k\in D_k}
	\frac{1}{n^k}\,
	\BBP(X_1^k\in A\big)\\
&=
	\frac{n!}{(n-k)!n^k}\,P_k(A).
\end{align*}

Now let $\clA=\{A_1,\ldots,A_N\}$ be any finite partition of $S$.
Applying the last bound above to each $A_i$ and summing over
$i=1,\ldots,N$,
$$\sum_{i=1}^NP_k(A_i)\log\frac{P_k(A_i)}{M_{k,\mu_n}(A_i)}
\leq \log\Big(\frac{n^k(n-k)!}{n!}\Big).
$$
Taking the supremum of this over all finite partitions $\clA$ and
recalling~\cite{pinsker:book} that 
the relative entropy between any two probability 
measures is equal to the supremum over 
all finite partitions $\clA$, completes the proof.
\end{proof}

\appendix
\section{Appendix }
%%%%%%%%%%%%%%%%%%%%%%%%%%%%%%%%%%%%%%%%%%%%%%%%%%%%%%%%%%%%%%%%%%%%%%%%%%

Let $X_1^n$ be an arbitrary (measurable) random vector. The measurability
of the empirical measure $\hat{P}_{X_1^n}$ mentioned in 
Section~\ref{s:dfTV} immediately follows from the following lemma.

\begin{lemma} \label{lemma:dirac}
The Dirac measures $\delta_s$ are measurable maps from $(S,\mathcal{S})$ to $(\mathcal{M}(S),\mathcal{F}),$ where 
$\mathcal{F}$ is the $\sigma$-algebra generated by the maps $\{\pi_A :\mathcal{M}(S) \to [0,1]\}_{A \in \mathcal{S}}$
given by $\pi_A(P) = P(A)$.
\end{lemma}
\begin{proof}
We need to show that $\{s: \delta_s \in F\} \in \mathcal{S}$ for every $F\in \mathcal{F}$. 

Since $\mathcal{F}$ is generated by the maps $\{\pi_A\}$, 
it is the smallest $\sigma$-algebra that contains all of,
$$\bigcup_{A \in \mathcal{S}}
\sigma\bigl(\{P: \pi_A(P) \in B\}_{ B \in \mathcal{B}([0,1])}\bigr),$$ 
where $\mathcal{B}([0,1])$ is the Borel $\sigma$-algebra on $[0,1]$. So it suffices to show that, for every $A \in \mathcal{S}$,
\begin{equation} \label{justpie}
 \sigma\bigl(\{P: \pi_A(P) \in B\}_{ B \in \mathcal{B}([0,1])}\bigr) 
\subset \mathcal{S}.
\end{equation}
Now we claim that, for every $A \in \mathcal{S}$, 
\begin{equation} 
\label{justjust}
\{s:\delta_s \in \{P: \pi_A(P) \in B\}\}_{ B \in \mathcal{B}([0,1])} 
\in \mathcal{S},
\end{equation}
implies~\eqref{justpie}.
But this follows since the collection of 
sets $F \in \mathcal{F}$ such that $\{s: \delta_s \in F\} \in \mathcal{S}$ 
forms a $\sigma$-algebra~\cite{williams:book}.

So it only remains to establish~\eqref{justjust}. 
Consider four cases:  
$(\{0,1\} \subset B),$ 
$(0 \in B, 1 \notin B),$
$(1 \in B, 0 \notin B),$
and $( \{0,1\} \in B^c)$.
For each of these cases, the set on the left-hand 
side of~\eqref{justjust} is simply $S$,  
$A^c$, $A$, and $\emptyset$, respectively,
all of which are in $\mathcal{S}$.

The result follows. 
\end{proof}

\bibliography{ik}
 
\end{document}